\documentclass[11pt]{article}

\usepackage{amsmath,amsthm,verbatim,amssymb,amsfonts,amscd, graphicx}
\usepackage{graphics}
\theoremstyle{plain}
\newtheorem{theorem}{Theorem}

\newtheorem{lemma}{Lemma}

\theoremstyle{definition}
\newtheorem{definition}{Definition}

\begin{document}

\title{Constructing Seifert surfaces from $n$-bridge link projections}
\author{Joan E. Licata}
\maketitle

\begin{abstract}
This paper presents a new algorithm $\mathfrak{A}$ for constructing Seifert surfaces from $n$-bridge projections of links.  The algorithm produces minimal complexity surfaces for large classes of braids and alternating links.  In addition, we consider a family  of  knots for which the canonical genus is strictly greater than the genus, ($g_c(K) > g(K)$), and show that $\mathfrak{A}$ builds surfaces realizing the knot genus $g(K)$.  We also present a generalization of Seifert's algorithm which may be used to construct surfaces representing arbitrary relative second homology classes in a link complement.
\end{abstract}

\section{Introduction}

 A Seifert surface for an oriented link is an orientable surface whose oriented boundary is the link.  This notion gives rise to a fundamental invariant of knots and links: the minimal genus of a Seifert surface for $L$ is known as the \textit{genus} of $L$ and denoted by $g(L)$.  In this paper we present a new algorithm, $\mathfrak{A}  $, which builds Seifert surfaces from $n$-bridge projections of links.  For certain large classes of link projections, the algorithm builds surfaces realizing the link genus.  (Theorems~\ref{thm:alt} and ~\ref{thm:braid}.) 

Seifert's algorithm is a classical method for constructing Seifert surfaces from a link projection.  The minimal possible genus of a surface constructed via Seifert's algorithm is known as the \textit{canonical genus} $g_c(L)$ of the link.  In the case of alternating projections or positive braids, Seifert's algorithm realizes both the genus and the canonical genus.   However, there also exist links for which $g_c(L) > g(L)$.   Section~\ref{sect:ex}  presents a family of such examples due to Kobayashi, Kobayashi, and Kawauchi (\cite{K} , \cite{Ka}),  and we show that $\mathfrak{A}$ successfully yields surfaces realizing the knot genus.  

\begin{theorem}\label{thm:alt}
The surface $\Sigma_{\mathfrak{A}}$ built by applying $\mathfrak{A}$ to an alternating $n$-bridge projection of a link $L$ has genus equal to the genus of $L$.
\end{theorem}

\begin{theorem}\label{thm:braid}
Let $B$ be a braid on $n$ strands with the property that each generator of the braid group appears with only one sign in the braid word.   If $L$ is the $n$-bridge link projection formed by taking the closure of $B$, then the surface $\Sigma_{\mathfrak{A}}$ built by applying $\mathfrak{A}$ to $L$ has genus equal to the genus of $L$
\end{theorem}

Morse theory provides the key tool in this construction.  An $n$-bridge projection of a link has a natural Morse function given by height on the page, and this extends to any Seifert surface for the link.  Given such a surface, the Euler characteristic is a signed sum of the number of critical points of the Morse function $F$.  These in turn correspond to changes in the topology of $F^{-1}(x_i)$ for generic $x_i \in \mathbb{R}$, where each $F^{-1}(x_i)$ can be thought of as a horizontal slice through the surface.  Although this description takes a Seifert surface as its starting point, the observation that a sequence of horizontal sections through a surface determines its Euler characteristic may be used constructively: by specifying such a sequence, one in fact builds a Seifert surface.  The algorithm  $\mathfrak{A}$ is a set of instructions for constructing a sequence of slices compatible with a fixed $n$-bridge link projection.  

The Morse theoretic approach to surfaces is fundamental to the new algorithm $\mathfrak{A}$, but this point of view also enables a new presentation of the classical Seifert's algorithm.  This in turn extends to an algorithm which constructs surfaces representing arbitrary second homology classes in link complements.  

\subsection{Organization and conventions}
Section~\ref{sect:alg} begins with a more detailed exposition of the Morse theoretical context, and $\mathfrak{A}$ is described in detail in \ref{sect:algA}.  Section~\ref{sect:ex} presents several examples of $\mathfrak{A}$ applied to links, including examples for which the canonical genus is greater than the genus.  A Morse-theoretic version of Seifert's algorithm, introduced in Section~\ref{sect:sa}, is used to prove Theorems~\ref{thm:alt} and \ref{thm:braid} in Sections \ref{sect:thm1} and \ref{sect:thm2}, respectively.  Finally, the paper concludes with a generalization of the Morse Seifert's algorithm to constructing surfaces representing arbitrary relative second homology classes in link complements.

Throughout this paper all links will be assumed to be nonsplit.   Links are also assumed to come with fixed projections, so that we speak of applying an algorithm to the link rather than to the link projection.  Finally, we will use a standard presentation of the braid group where generators are positive half-twists between adjacent strands:
 \[
B_k = ( b_1, b_2, ...b_{k-1} \| b_ib_{i+1}b_i=b_{i+1}b_ib_{i+1};  b_i b_j = b_j b_i,  |i-j| \ge 2 )
\]

\section{The algorithm $\mathfrak{A}$}\label{sect:alg}

Thurston introduced the notion of the \textit{complexity} of a surface embedded in a three-manfold \cite{T}.

\begin{definition}
Let $S$ be a surface with components $s_i$ properly embedded in a three-manifold with bounday.  The \textit{complexity} of the surface, $\chi_-(S)$, is given by the following sum, where $\chi(s_i)$ denotes the Euler characteristic of $s_i$ :
\[\chi_-(S)=\sum_{ i \colon \chi(s_i) \le 0 } -\chi (s_i)\]
.
\end{definition}

A minimal complexity surface is a minimal genus surface, but complexity is both easier to work with in this context and generalizes more naturally to other second homology classes in the link complement. (See Section~\ref{sect:arb}.)

\subsection{Morse theory preliminaries}
\begin{definition} An  \textit{upper crossingless match} on $2n$ colinear points is a collection of $n$ disjoint curves bounded by the points and lying entirely above the line of the points.  The vertical reflection of such a figure, in which all the curves lie below the line of the points, is a \textit{lower crossingless match}.
\end{definition}

\begin{figure}
\begin{center}
\scalebox{.60}{\includegraphics{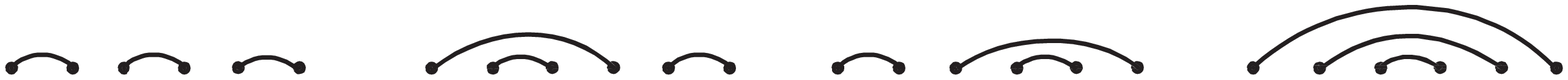}}
\end{center}
\caption{The upper crossingless matches on six points.}\label{fig:cm}
\end{figure}

Figure~\ref{fig:cm} illustrates the upper crossingless matches on six points.

An $n$\textit{-bridge projection} of a link is a decomposition into a pair of crossingless matches on $2n$ points (one upper and one lower) together with a braid on $2n$ strands connecting the endpoints of the crossingless match curves.  The curves of the crossingless matches are referred to as the \textit{bridges} of the projection.  Every link has an $n$-bridge projection for some $n$.  This is easily seen via the fact that any link has a projection as a closed braid, which is a $k$-bridge projection for $k$ equal to the braid index.  Given any $n$-bridge  projection, height on the page provides a natural Morse function $f: L \rightarrow [0,1]$  which maps the upper and lower bridges to $[1, 1-\epsilon)$ and $(\epsilon, 0]$, respectively.  

For any $x \in (\epsilon, 1-\epsilon)$, $f^{-1}(x)$ consists of $2n$ points which may be thought of as the intersection of the braid $B$ with a horizontal plane at height $x$.  As $x$ changes, the points move around this plane via an isotopy defined by $B$.  More precisely, the braid group generator $b_i$ maps to an isotopy $\mathfrak{b}_i$ of the marked plane which interchanges the $i^{th}$ and $(i+1)^{th}$ points  by clockwise rotation of a neighborhood of the pair.   (See Figure~\ref{fig:isot}.) $\mathfrak{b}_i$ is the identity away from a neighborhood of these two points, and multiplication in the braid group corresponds to composition of the associated $\mathfrak{b}_i$.   This map is well-defined up to isotopy fixing the endpoints of $B$.     
 
 \begin{figure}
\begin{center}
\scalebox{.60}{\includegraphics{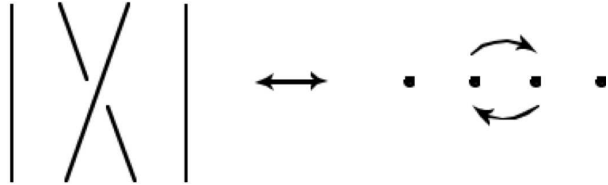}}
\end{center}
\caption{The positive crossing on the left corresponds to the isotopy of the marked plane shown on the right.}\label{fig:isot}
\end{figure}

The Morse function $f$ extends to a Morse function $F$ on any Seifert surface $\Sigma$ for $L$.  For a generic $x$, $F^{-1}(x)$ is a collection of $n$ disjoint curves connecting the points of $f^{-1}(x)$ and possibly some  disjoint simple closed curves.  We call the pre-image of a generic point of $F$ a \textit{frame}.  According to the Morse Lemma, (\cite{Mi}, Theorem 3.2), a critical point for $F$ corresponds to handle addition in $\Sigma$, and this in turn is reflected in a change in the topology between frames immediately above and below the critical point.  

There are three types of handle additions, corresponding to index zero, one, and two critical points.   Adding a zero-handle (``birth move") to $\Sigma$ corresponds to introducing a new innermost simple closed curve to a frame.  Adding a two-handle reverses this process and is known as a death move.  One-handle addition creates a saddle in $\Sigma$ and changes a frame by resolving along an arc as indicated in Figure~\ref{fig:crit}.  

To build a surface using $\mathfrak{A}$, let  $F^{-1}(1-\epsilon)$ be the upper crossingless match of the link projection.  The braid isotopy described above acts on this frame, and $\mathfrak{A}$ specifies saddle resolutions to perform. Each resolution increases the complexity of $\Sigma$ by one.  When the entire braid isotopy has acted, further resolutions are performed until the curves in the resulting frame may be paired with the lower crossingless match to create simple closed curves.  These curves are capped off using death moves, each of which decreases the complexity of $\Sigma$ by one.  Ultimately the Euler characteristic of the surface constructed is equal to the number of death moves perfomed minus the number of saddle resolutions.

\begin{figure}
\begin{center}
\scalebox{.45}{\includegraphics{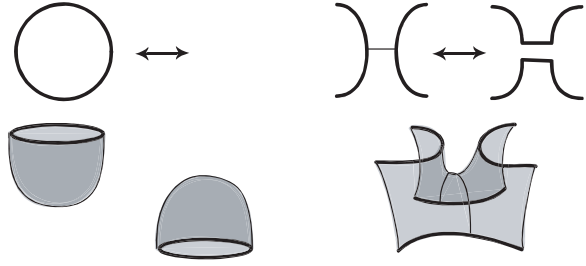}}
\end{center}
\caption{Left: Adding or subtracting an innermost simple closed curve corresponds to a zero-handle or two-handle addition in the surface $\Sigma$. Right: Resolution along an arc connecting compatibly-oriented curves corresponds to a saddle in $\Sigma$. }\label{fig:crit}
\end{figure}

\subsection{The algorithm $\mathfrak{A}$}\label{sect:algA}
The algorithm $\mathfrak{A}$ consists of the precise instructions for which saddle resolutions to perform, together with an end-state algorithm which eventually permits the curves in some frame to be glued to the lower bridges and capped off. This section describes the resolutions performed while the isotopy acts, deferring the end-state algorithm to Section~\ref{sect:esa}.  The reader may also find it helpful to look at the examples in Section~\ref{sect:ex}.

There are several definitions which will be convenient for describing $\mathfrak{A}$:

\begin{definition} Two curves in a crossingless match are \textit{stacked} if the endpoints of one curve are between the endpoints of the other.  The latter curve is the \textit{outer} curve, and the former, the \textit{inner} curve.
\end{definition}

In Figure~\ref{fig:cm}, the three crossingless matches on the right have stacks.  In general, a frame will consist of $n$ curves with colinear endpoints, but it will not be a crossingless match.  

\begin{definition} 
An endpoint of a curve is $obstructed$ if some curve in the frame passes beneath it.
\end{definition}

Note that a frame with no obstructions is an upper crossingless match.  We will also distinguish between \textit{direct} and \textit{indirect} obstruction: an endpoint is directly obstructed  by the curve immediately below it, and indirectly obstructed by any curves below that.  The \textit{obstruction number} of a frame is the sum over the endpoints of the number of curves below each point.

\begin{definition}
A curve connecting two points in a frame is \textit{critical} if it is not isotopic, in the complement of the $2n$ endpoints, to a curve with no vertical tangent line.  A frame is \textit{critical} if any of its constituent curves are critical.
\end{definition}

In a neighborhood of the upper bridges, the surface looks like the product of the upper crossingless match with an interval, so $F^{-1}(1-\epsilon)$ is a copy of this crossingless match.  Note that these curves inherit a transverse orientation from the link.  If the action of the first elementary isotopy $\mathfrak{b}_i$ on $F^{-1}(1-\epsilon)$ creates a frame with an obstruction, identify a saddle resolution to eliminate it as indicated in Figure~\ref{fig:obstrarc}.  The obstruction arcs shown there indicate resolutions consistent with the orientations of the curves, ensuring that the resulting surface will be orientable.  

\begin{figure}
\begin{center}
\scalebox{.45}{\includegraphics{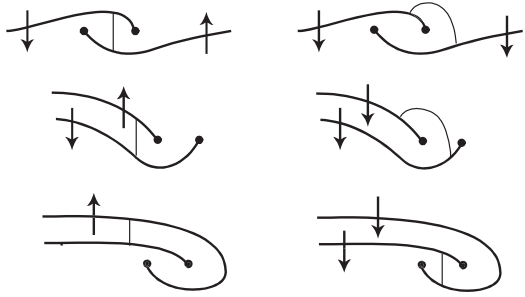}}
\end{center}
\caption{In each figure, the thin line indicates an arc whose resolution removes the obstruction. }\label{fig:obstrarc}
\end{figure}

Once the obstruction arc is identified, let the next elementary braid isotopy act on the decorated diagram and add new obstruction arcs if direct obstructions are created.  If a curve ceases to directly obstruct a particular endpoint, remove the corresponding resolution arc.   Continue decorating the diagram with obstruction arcs and letting the isotopy act until a frame becomes critical.   When this happens, resolve the obstruction arcs which indicate obstructions by the critical curve(s).  Once these are resolved, it may be necessary to add additional obstruction arcs so that all obstructions are marked.  If the frame had indirect obstructions, some curves may still be critical and the iredecorate/resolve process might need to be repeated.  Once the frame becomes non-critical and all obstructions are marked with appropriate arcs, let the next elementary isotopy act.  Repeat this process until the braid-induced isotopy is exhausted.  If the final frame is a crossingless match, glue the endpoints to the corresponding points of the lower bridges and remove the simple closed curves.  If the final frame is obstructed, apply the end-state algorithm described in Section~\ref{sect:esa} before gluing to the lower bridges.

\subsection{Examples}\label{sect:ex}
This section contains two examples of minimal complexity surfaces constructed using $\mathfrak{A}$.  The first example (Figure~\ref{fig:Kob}) shows $\mathfrak{A}$ applied to a family of knots whose members are distinguished by a twist parameter.  This family was studied by Kobayashi, Kobayashi, and Kawauchi, who showed that when $n \ne 0, 12$, the knot has $g(K)<g_c(K)$.   \cite{K}, \cite{Ka}

The second example (Figure~\ref{fig:alf}) is a knot studied by Alford, who showed that it has two minimal complexity surfaces with non-homeomorphic complements (\cite{A}).  The choice imposed by $\mathfrak{A}$ to make $F^{-1}(1-\epsilon)$ agree with the upper crossingless match determines the surface shown here.  One may represent any Seifert surface via a sequence of frames, however, so we note that Alford's other Seifert surface would have an initial frame in which additional simple closed curves enclose components of the crossingless match.

\begin{figure}
\begin{center}
\scalebox{.6}{\includegraphics{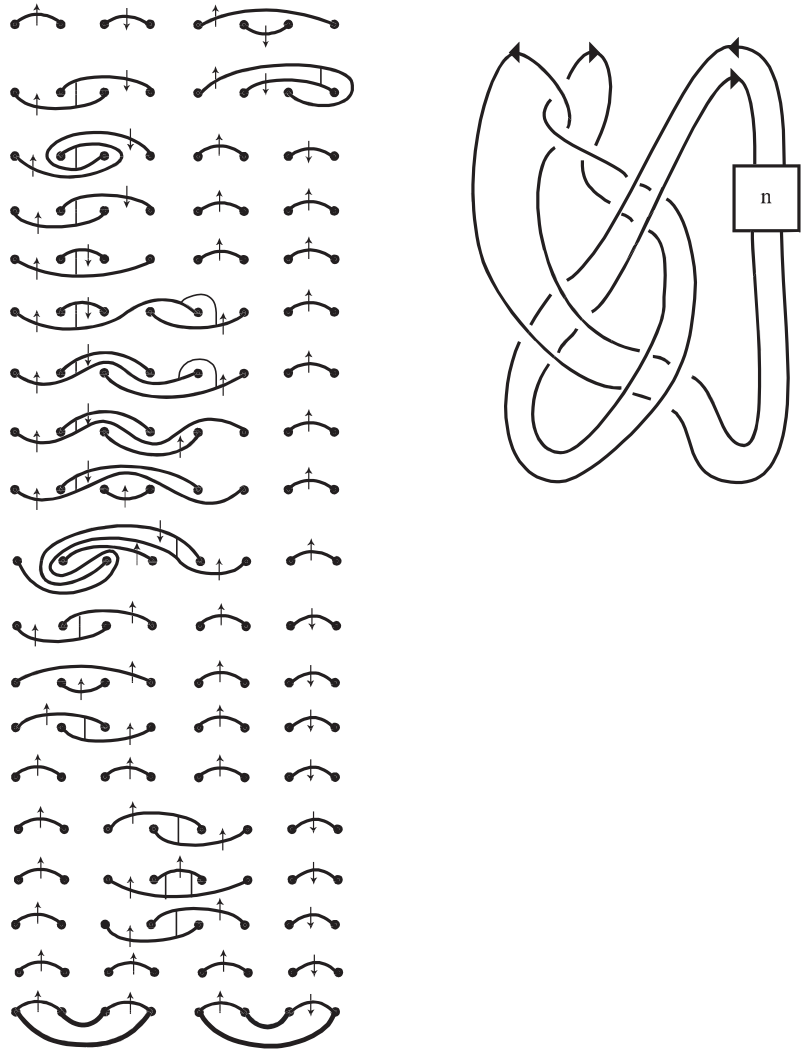}}  
 \end{center}
\caption{The box represents $n$ full twists, where $n \ne 0,12$, and the surface shown realizes $g(K)$.}\label {fig:Kob}
\end{figure}

\begin{figure}
\begin{center}
\scalebox{.6}{\includegraphics{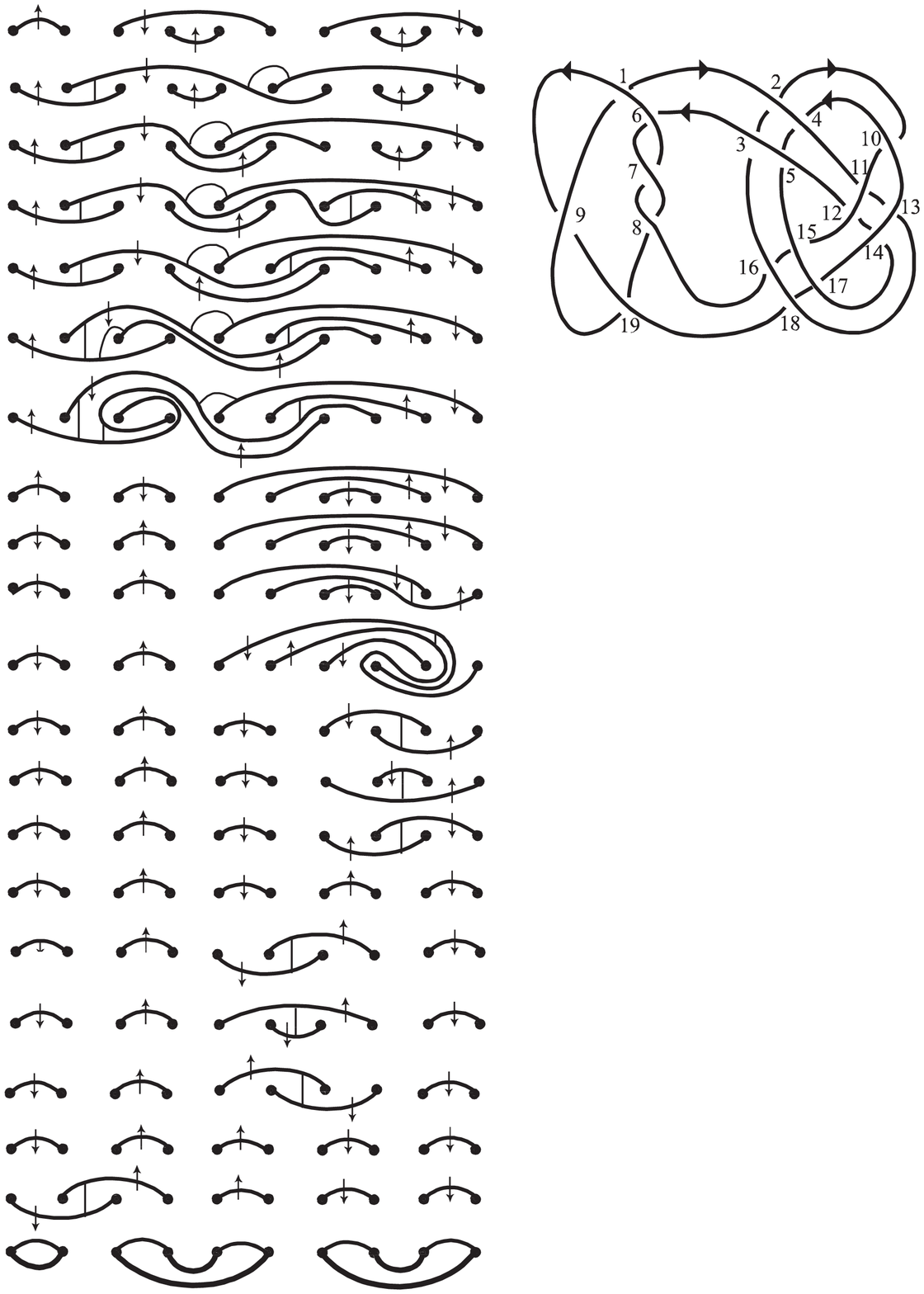}}  
 \end{center}
\caption{$\mathfrak{A}$ applied to a knot with two minimal complexity surfaces.  The numbers on the knot diagram indicate the order in which the corresponding crossings appear in the braid word used in the construction.}\label {fig:alf}
\end{figure}

\subsection{End-state algorithm}\label{sect:esa}

In order to produce a Seifert surface for the link $L$, the final frame must be glued to the lower bridges and the resulting simple closed curves capped off with two-handles.  This gluing operation is always possible if the final frame is a crossingless match.  The final ingredient in $\mathfrak{A}$ is an algorithm for resolving an arbitrary non-critical frame to a crossingless match. 

\begin{definition}
Two curves are $parallel$ if no part of another curve comes between them.
\end{definition}

\begin{definition}
An \textit{unstacking arc} is a resolution arc lying between and connecting two parallel curves with opposite transverse orientations.
\end{definition}

The algorithm to change an arbitrary non-critical frame into a crossingless match is depicted in (\ref{pict:esa}).  Every resolution step either reduces the obstruction number of the frame or shortens the total length of the curves, so the algorithm must terminate in finite time.  Once the frame has been transformed into a crossingless match, the curves are glued to the lower bridges of the link projection and the resulting simple closed curves capped off with two-handles.

\begin{picture}(400,310)
\put (220, 295){I. Are there any indirectly}
\put(220, 285){obstructed endpoints?}
\put(235,280){\vector(-3,-2){35}}
\put(320,280){\vector(3,-2){35}}
\put(120, 250){Is a single point}
\put(120, 240){obstructed by curves with}
\put(120, 230){opposite co-orientations?}
\put(175,225){\vector(-3,-2){35}}
\put(230,225){\vector(3,-2){35}}
\put(350,240){Resolve all arcs}
\put(380, 235){\vector(0,-1){30}}
\put(370, 190){CM}
\put(70,190){II. Is there a pair of parallel}
\put(70, 180){ obstructing curves with} 
\put(70, 170){opposite co-orientations?}
\put(100,165){\vector(-3,-2){30}}
\put(110,165){\vector(3,-2){35}}
\put(270,190){Resolve all arcs}
\put(305,185){\vector(0,-1){25}}
\put(270, 150){$\exists$ obstructions?}
\put(300,145){\vector(-3,-2){35}}
\put(310,145){\vector(3,-2){35}}
\put(250, 110){Relabel \& }
\put(250, 100){return to I}
\put(340, 110){CM}
\put(110, 130){Pick a bottommost}
\put(110, 120){pair of obstructing}
\put(110, 110){curves and resolve all.}
\put(110, 100){arcs from the lower curve.}
\put(145,95){\vector(0,-1){25}}
\put(110, 60){Return to II.}
\put(20, 130){Resolve the}
\put(20, 120){unstacking arc}
\put(50,115){\vector(0,-1){25}}
\put(20, 80){$\exists$ obstructions?}
\put(50,75){\vector(-3,-2){30}}
\put(50,75){\vector(3,-2){35}}
\put(0, 35){Relabel \&}
\put(0, 25){return to I}
\put(80, 35){CM}

\put(240, 45){All``Yes" branches are to the left,} 
\put(240, 35){and all ``No" branches to the right.}

\end{picture}

\begin{equation}\label{pict:esa}
The\ end-state\ algorithm
\end{equation}

\begin{figure}
\begin{center}
\scalebox{.5}{\includegraphics{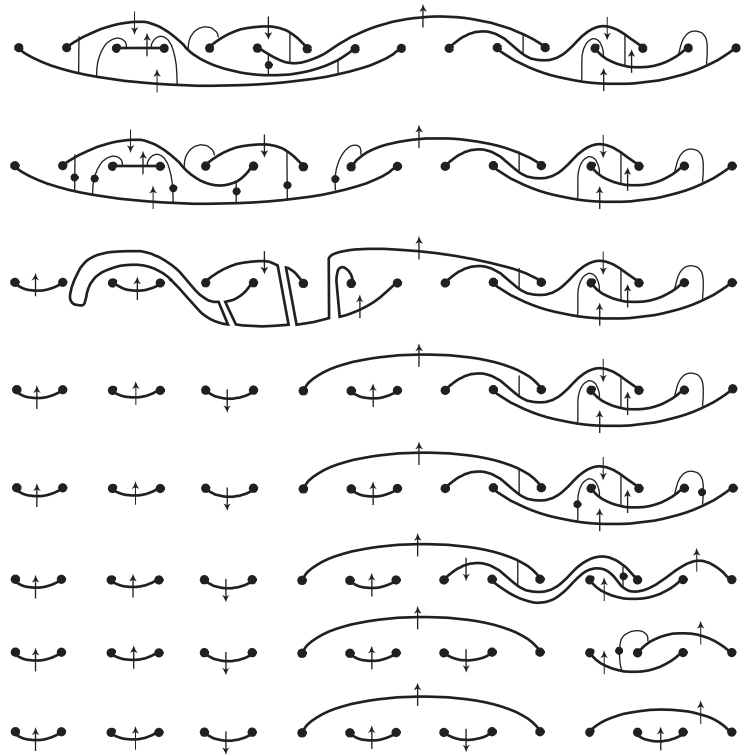}}  
 \end{center}
\caption{The end-state algorithm applied to a complicated frame.  Dots indicate which resolution arcs in a given frame are being resolved.}\label {fig:esa}
\end{figure}

\section{A Morse version of Seifert's algorithm}\label{sect:sa}

In order to prove the minimal complexity results claimed for $\mathfrak{A}$, we will compare the surfaces built using this technique to those built using Seifert's algorithm.  We thus introduce a variant of $\mathfrak{A}$, denoted $\mathfrak{A}_{\mathfrak{S}}$, which we prove is equivalent to Seifert's algorithm.  
 
$\mathfrak{A}_{\mathfrak{S}}$ is similar in structure to $\mathfrak{A}$, but resolves obstructions more agressively.  Beginning with an $n$-bridge projection, let the first elementary isotopy act and decorate the resulting frame with obstruction arcs as in $\mathfrak{A}$.  Immediately resolve all the obstruction arcs to get a (possibly different) crossingless match.  Repeat this process, always resolving a frame to a crossingless match before letting the next elementary isotopy act.  When the isotopy is exhausted, glue the crossingless match in the final frame to the lower bridges and cap off all simple closed curves with two-handles. 

\begin{theorem}\label{thm:sa}
When $\mathfrak{A}_{\mathfrak{S}}$ and Seifert's algorithm are applied to the same projection, they produce isotopic surfaces.
\end{theorem} 

\begin{proof}
Fix a projection and denote the surface constructed by $\mathfrak{A}_{\mathfrak{S}}$ by $\Sigma_{\mathfrak{A_{\mathfrak{S}}}}$.  Similarly, use $\Sigma_{SA}$ to denote the surface constructed by Seifert's algorithm.

$\Sigma_{\mathfrak{A_{\mathfrak{S}}}}$ decomposes naturally into \textit{ simple subsurfaces}, each corresponding to a single elementary isotopy.  The first frame in each simple subsurface is a crossingless match.  The elementary isotopy corresponding to a single braid generator acts on this frame, and any resulting obstructions are removed using $\mathfrak{A}_{\mathfrak{S}}$.  This leaves a (possibly different) crossingless match for the final frame, which is also the initial frame for the next simple subsurface.   

Since both $\Sigma_{\mathfrak{A}_{\mathfrak{S}}}$ and $\Sigma_{SA}$ are built from the same projection,  one may compare the corresponding subsurfaces when $\Sigma_{SA}$ is cut at the same heights on the projection.   Recall that a frame is the preimage of a point under the height Morse function from a Seifert surface to $\mathbb{R}$; cutting $\Sigma_{SA}$ at height $x$ thus yields a frame which depicts the intersection of a horizontal plane with $\Sigma_{SA}$.  (Our term ``stacking" arises from this viewpoint, where stacked curves in a frame correspond to concentric Seifert cycles, and thus vertically stacked Seifert discs.)   A simple subsurface of $\Sigma_{SA}$ is determined by the initial crossingless match, which is inherited from the previous subsurface, and the orientations of the crossing strands.  

Up to reflection, $\mathfrak{A}_{\mathfrak{S}}$ distinguishes seven types of crossings based on whether the crossing strands are the endpoints of stacked (or unstacked) curves with the same (or opposite) transverse orientations.  Seifert's algorithm distinguishes only between crossing strands with the same or opposite orientations.  Figure~\ref{fig:subsurf} shows what each  constructions yields in each of the cases.  It is important to note that the upper crossingless match is inherited from the subsurface above it; in the absence of this information, the lower crossingless match is not uniquely determined.  Note that the non-crossing strands shown in the diagrams might not be adjacent to the crossing strands, but the isotopy types of the subsurfaces are unaffected by disjoint Seifert discs or stacked curves.

\begin{figure}
\begin{center}
\scalebox{.6}{\includegraphics{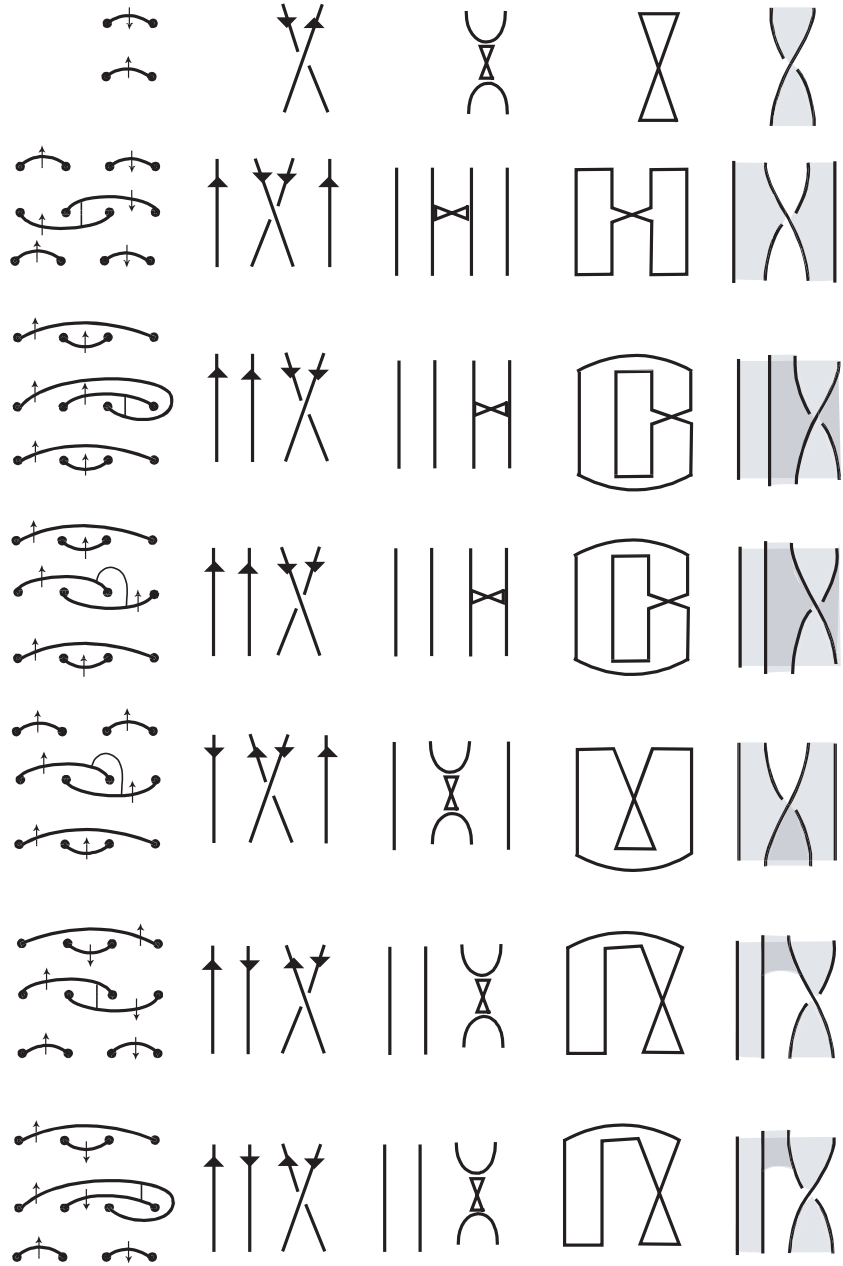}}  
 \end{center}
\caption{Left: Frames for a simple subsurface of $\mathfrak{A}_{\mathfrak{S}}$.  Second column: The subbraid corresponding to the simple subsurface.  Third column:  Seifert's algorithm resolves each crossing according to the orientation of the braid strands and glues in a twisted band.  Fourth and fifth columns: Schematic and pictorial illustrations of the simple subsurfaces of  $\Sigma_{SA}$. }\label {fig:subsurf}
\end{figure}

Comparing the left and right columns on Figure~\ref{fig:subsurf}, one sees that for a fixed crossing type and initial crossingless match, the simple subsurfaces of $\Sigma_{\mathfrak{A}_{\mathfrak{S}}}$ and $\Sigma_{SA}$ are isotopic.   In particular, they each produce the same final crossingless match, so the proof is inductive on the number of cuts.  The base case is the top crossingless match, where $\Sigma_{\mathfrak{A}_{\mathfrak{S}}}$ and $\Sigma_{SA}$ agree, and the inductive step shows that if the two surfaces agree at the $n^{th}$ cut, they agree at the $n+1^{th}$ cut as well.  

 \end{proof}

\section{Proof of Theorem~\ref{thm:alt}}\label{sect:thm1}

For an alternating projection, we claim that $\mathfrak{A}$ builds surfaces of the same complexity as does $\mathfrak{A}_{\mathfrak{S}}$.  Theorem~\ref{thm:sa} establishes that $\Sigma_{\mathfrak{A}_{\mathfrak{S}}  }$ and $\Sigma_{SA}$ are isotopic, and Seifert's algorithm is known to construct a minimal complexity surface from any alternating projection.

 \begin{lemma}\label{lem:obstr} 
In an alternating projection, no obstruction will vanish under the action of the braid isotopy.
\end{lemma}

Thus, once an endpoint in a frame is obstructed, some resolution will be required to remove this obstruction.  

\begin{proof}[Proof of Lemma~\ref{lem:obstr}]  Beginning with a non-critical frame, a non-critical obstruction occurs when the right (respectively, left) endpoint of a curve passes in front of the adjacent endpoint to its left,  (right).  Without loss of generality, consider the first case.  Number the points from left to right, so that $p_i$ passes in front of $p_{i+1}$ in the situation described.  In order for this obstruction to vanish without the frame becoming critical, either the over-crossing endpoint must retreat again to the left, or else the entire curve obstructing curve must pass in front of $p_{k+1}$. In an alternating projection, neither of these is possible.

An alternating braid has the property that each generator may appear in the braid word with exactly one sign.  Furthermore, if $b_i$ and $b_{i+1}$ each occur in the braid word (with some sign) at least once, then in fact they must occur with opposite signs.  The over-crossing of the right endpoint $p_i$ corresponds to $b_i^{-1}$, and since $b_i $ doesn't appear in the braid word, the only way to remove the obstruction is for the entire curve bounded by $p_i$ to move to the right of $p_{i+1}$.  (Note that we are retaining the labels on the points dictated by their positions when the obstruction first occurred.) 

The condition that neighboring braid generators alternate in sign restricts the kind of frames that can evolve from a crossingless match via the braid isotopy.   If an elementary isotopy switches  the endpoints of a single curve, the transverse orientation changes, but the (unoriented) crossingless match remains the same.  Any other elementary isotopy, however, has the property that when it acts twice, the frame becomes critical.  Thus, disregarding orientation, there is only a small list of frames that may occur before any further elementary isotopies make the frame critical.   See Figure~\ref{fig:wave} for an example of the frames which may occur when $n=3$; the distinctive wave-like appearance of such frames makes it easy to identify whether a given frame is in this set even for large $n$.  If a single elementary isotopy renders one of these frames critical, the resolution(s) performed under $\mathfrak{A}$ returns a frame from this list.  Thus, in the course of applying $\mathfrak{A}$ to an alternating projection, the only non-critical frames which occur are in fact frames that could have come from some crossingless match without the frame first becoming critical.  

 \begin{figure}
\begin{center}
\scalebox{.6}{\includegraphics{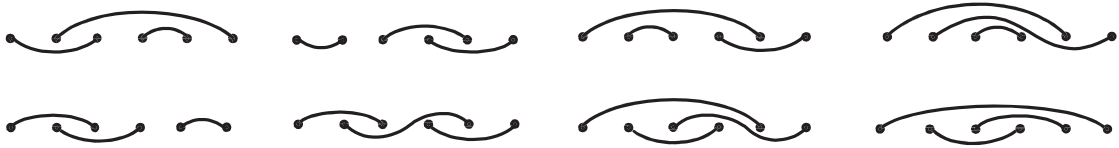}}  
 \end{center}
\caption{A complete collection of the non-critical frames which can evolve from a crossingless match on six points under the action of a braid involving the generators $b_1^{-1}, b_2, b_3^{-1}, b_4$, and $b_5^{-1}$.}\label{fig:wave}
\end{figure}

One consequence of this argument is that there are no indirect obstructions in non-critical frames.
This constraint allows us to show that the left endpoint, $p_k$, of the curve obstructing $p_{i+1}$ cannot pass in front of $p_{i+1}$ without making a critical frame.  First, note that $p_{i+1}$ cannot move further to the left; since $b_i^{-1}$ appears in the braid word, $b_{i-1}$ may as well, but the corresponding isotopy passes $p_{i-1}$ behind $p_{i+1}$.  When $p_{i+1}$ is the left endpoint of a curve,  this creates an indirect obstruction of $p_{i-1}$ by the curve bounded by $p_i$ and $p_k$.  When $p_{i+1}$ is the right endpoint of a curve, the argument above restricts this curve to passing above $p_{i-1}$, and the elementary isotopy $b_{i-1}$ makes the frame critical.  (See Figure~\ref{fig:ind}.)
 
 \begin{figure}
\begin{center}
\scalebox{.6}{\includegraphics{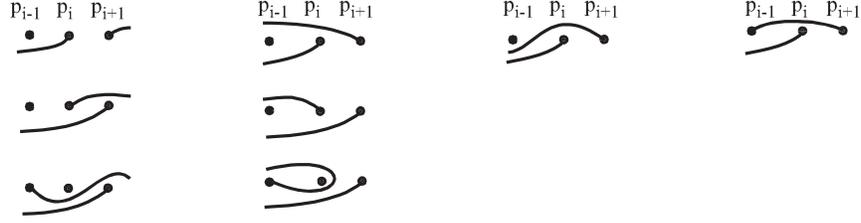}}  
 \end{center}
\caption{Left: If $p_{i+1}$ is a left endpoint, $\mathfrak{b}_{i-1}$  creates an indirect obstruction.  Center left: If $p_{i+1}$ is a right endpoint of a curve, $\mathfrak{b}_{i-1}$ makes the frame critical.  Center right: $p_{i+1}$ cannot be the right endpoint of a curve creating an indirect obstruction.  Right:  Since $b_{i-1}^{-1}$ does not appear in the braid word, this frame cannot have evolved from a crossingless match.}\label{fig:ind}
\end{figure}

  \begin{figure}
\begin{center}
\scalebox{.6}{\includegraphics{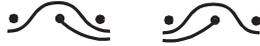}}  
 \end{center}
\caption{A curve passing alternately over and under endpoints will necessarily create indirect obstruction if its endpoints are separated by at least three points.}\label {fig:unob}
\end{figure}

With the position of $p_{i+1}$ fixed, note that for $p_k$ to become adjacent to $p_{i+1}$ without the frame becoming critical requires $p_{k}$ to move alternately in front of and behind the points separating it from $p_{i+1}$.   Requiring such crossings to preserve the non-criticality of the frame forces the curve connecting $p_{k}$ and $p_{i}$ to snake alternately over and under the separating points.  This implies that $i-k \le3$, for if the curve passes under a point, over a point, and under a second point, the curve bounded by the middle point indirectly obstructs one of the lateral ones.  (See Figure~\ref{fig:unob}).  For each of the cases when  $i-k \le3$, however, studying the possible frames directly shows that $p_k$ could only pass behind $p_{i+1}$.

\end{proof}

 \begin{lemma}\label{lem:agree} If $P$ is an alternating projection of a link, then the complexity of $\Sigma_{\mathfrak{A}}$ agrees with that of $\Sigma_{\mathfrak{A}_{\mathfrak{S}}  }$.
\end{lemma}

\begin{proof}[Proof of Lemma~\ref{lem:agree}] 
The proof rests on a process which interpolates between $\mathfrak{A}_{\mathfrak{S}}  $ and $\mathfrak{A}$ to  build a surface whose Euler characteristic agrees with both $\Sigma_{\mathfrak{A}}$ and $\Sigma_{\mathfrak{A}_{\mathfrak{S}}  }$.  Beginning with a crossingless match, let the braid isotopy act on the initial frame.  Decorate the resulting frame with obstruction arcs as if applying $\mathfrak{A}_{\mathfrak{S}}  $.  By hypothesis and Lemma~\ref{lem:obstr}, once an obstruction arc appears on the diagram, it cannot vanish under isotopy alone.  Continue the isotopy/decoration process until applying the next elementary isotopy would create a critical frame.  Before letting this isotopy act, first resolve any obstruction arcs such that the obstructing curve would become critical in the next frame. Comparing this frame to the corresponding frame had we applied $\mathfrak{A}_{\mathfrak{S}}  $, note that every resolution performed under $\mathfrak{A}_{\mathfrak{S}}  $ corresponds either to an obstruction arc still on the frame or an obstruction arc that was just resolved.  Now apply the next elementary isotopy  and consider the resulting frame.  If it is non-critical, proceed as above with the isotopy/decoration process.  If the frame is critical, then a new obstruction was created by a now-critical curve, and both $\mathfrak{A}_{\mathfrak{S}}  $ and $\mathfrak{A}$ require its resolution.  (For example, consider either of the bottom frames in Figure~\ref{fig:obstrarc}.)
 Repeat this procedure until the braid isotopy is exhausted.  Every obstruction arc added to the diagram was resolved  to prevent a critical frame, resolved to remove a criticality, or remains on the diagram.  

Note that applying $\mathfrak{A}$ to an alternating projection creates no frames with indirect obstruction.  Thus, the end-state algorithm in this case consists simply of resolving every arc on the frame, as if applying $\mathfrak{A}_{\mathfrak{S}}  $.  This completes a surface whose complexity agrees with that of $\Sigma_{\mathfrak{A}_{\mathfrak{S}}  }$, since the same set of arcs were resolved to give the same crossingless match.  Similarly, the surface is isotopic to $\Sigma_{\mathfrak{A}}$ since the two surfaces differ only by the order of the elements in elementary isotopy/resolution pairs.  
\end{proof}

\section{Proof of Theorem~\ref{thm:braid}}\label{sect:thm2}

The proof of Theorem~\ref{thm:braid} is similar to that of Theorem~\ref{thm:alt} in the sense that the argument again relies on comparing $\mathfrak{A}$ to $\mathfrak{A}_{\mathfrak{S}}  $.  However, we use sutured manifold theory to show that $\Sigma_{\mathfrak{A}_{\mathfrak{S}}  }$ realizes the genus of the link.  This result encompasses the known statement that Seifert's algorithm builds minimal complexity surfaces for positive braids.

\subsection{$\chi(\Sigma_{\mathfrak{A}_{\mathfrak{S}}  })=\chi(\Sigma_{\mathfrak{A}})$}
\begin{lemma} Let $L$ be the closure of a braid with the property that each braid generator appears with only one sign.  The surfaces $\Sigma_{\mathfrak{A}_{\mathfrak{S}}  }$ and $\Sigma_{\mathfrak{A}}$ have the same complexity.
\end{lemma}

\begin{proof}
Lemma~\ref{lem:obstr} states that for alternating projections, obstructions can be removed only by resolution.  However, with minor modification, the lemma may be extended to braids of the type covered by Theorem~\ref{thm:braid}.  These braids are characterized by the property that each braid group generator appears with only a single sign in the braid word; thus, removing an obstruction without resolution requires the entire obstructing curve to pass in front of the obstructed point.  For a closed braid projection this is impossible, as the right endpoints of the curves are all fixed.   This establishes the analogue of Lemma~\ref{lem:obstr} for the braids of Theorem~\ref{thm:braid}.   

Lemma~\ref{lem:agree} also adapts to these braids. The proof is identical in this case until the end-state algorithm is applied.  Although frames with indirect obstruction may occur when $\mathfrak{A}$ is applied to these projections, any two curves obstructing the same endpoint have the same co-orientation.  Therefore the end-state algorithm resolves all obstruction arcs simultaneously.  
\end{proof}

It is worth noting that the absence of unstacking arcs is necessary: in general, resolving a single unstacking arc may remove more than one obstruction, so resolving all obstruction arcs in that case could yield a surface with greater complexity than that of $\Sigma_{\mathfrak{A}_{\mathfrak{S}}  }$.  (This is illustrated by the first resolution in Figure~\ref{fig:esa}.)

\subsection{Sutured manifold theory}

Sutured manifold theory, developed by Gabai, is a useful tool for studying minimal complexity surfaces.  This section collects some basic results, and we refer the reader to \cite{G1} for more details. 

\begin{definition}
A \textit{sutured  manifold} $(M, \gamma)$ is a compact, oriented three-manifold with a collection of distinguished annuli and tori in its boundary.  Furthermore, each annular component of $\gamma$ is equipped with an oriented core core curve $s(\gamma)$ known as the \textit{suture}.
\end{definition}

\begin{definition}
If $(M, \gamma)$ is a sutured manifold, let $R(\gamma) = \partial M - \gamma$. 
\end{definition}

Components of $R(\gamma)$ are coherently transversely-oriented in the sense that the boundary of each component of $R(\gamma)$ represents the same homology class in $H_1(\gamma)$ as some suture.  Cutting a sutured manifold along an oriented surface $(S, \partial S)$ embedded in $(M, \gamma)$ gives rise to a canonical $\gamma'$ on the cut manifold $M'$.  In particular, the new components of $R(\gamma')$ inherit their orientations from that of $S$, and new sutures arise where oppositely-oriented components of $R(\gamma')$ meet.  

\begin{definition}\label{def:taut}
A sutured manifold $(M, \gamma)$ is $taut$ if $M$ is irreducible and $R(\gamma)$ is norm-minimizing in $H_2(M, \gamma)$. 
\end{definition}

\begin{lemma}[\cite{G1}, Lemma 3.12]\label{lem:disc}
Let $(M, \gamma) \rightarrow (M', \gamma')$ be a decomposition along a disc $J$ such that $|J \cap s(\gamma)|=2$.  Then $(M, \gamma)$ is taut if and only if $(M', \gamma')$ is taut.
\end{lemma}

One may view a Seifert surface $\Sigma$ for a link as a properly embedded surface in the link complement.  Cutting the link complement along this surface induces a sutured manifold structure where the two copies of $\Sigma$ become the components of $R(\gamma)$.  The boundary of the neighborhood  of each link component becomes an annulus in the cut manifold, and the cores of these annuli are the sutures.  In this setting, showing that the original surface $\Sigma$ was of minimal complexity is equivalent to showing the sutured manifold is taut.

\subsection{$\mathfrak{A}_{\mathfrak{S}}  $ produces taut sutured manifolds}
\begin{proof}[Proof of Theorem~\ref{thm:braid}]

To show that $\mathfrak{A}_{\mathfrak{S}} $ is a minimal complexity surface, consider the sutured manifold $M$ produced by cutting the link complement along $\mathfrak{A}_{\mathfrak{S}}$.  Specifically, we study the complementary sutured manifold, $M^c$, which is a product neighborhood of $\mathfrak{A}_{\mathfrak{S}} $ with the same sutures as $M$.  A disc decomposition in $M$ is equivalent to gluing a two handle to $M^c$.

If there were no crossings in the braid, every frame of $\mathfrak{A}_{\mathfrak{S}} $ would consist of a single stack of $n$ curves, so $M^c$ would consist of $n$ balls, each with a single simple closed curve suture.  Each saddle resolution in $\mathfrak{A}_{\mathfrak{S}} $ introduces a one-handle connecting two of these balls, and these are shown in  Figure~\ref{fig:suture}.  Note that the direction the sutures rotate around the one-handle is dictated by whether the corrseponding crossing in the braid is positive and negative.  We thus distinguish between positive one-handles and negative one-handles.  

\begin{figure}
\begin{center}
\scalebox{.6}{\includegraphics{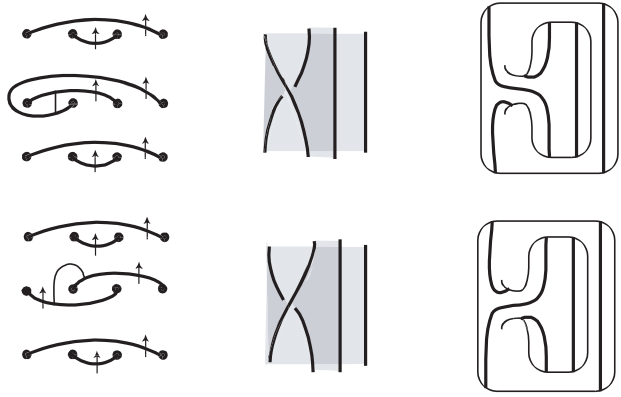}}  
 \end{center}
\caption{Top: A negative crossing.  Bottom: A positive crossing.  The righthand figures show two three-balls connected by a one-handle, with the sutures indicated by the darker lines.  }\label {fig:suture}
\end{figure}

\begin{figure}
\begin{center}
\scalebox{.6}{\includegraphics{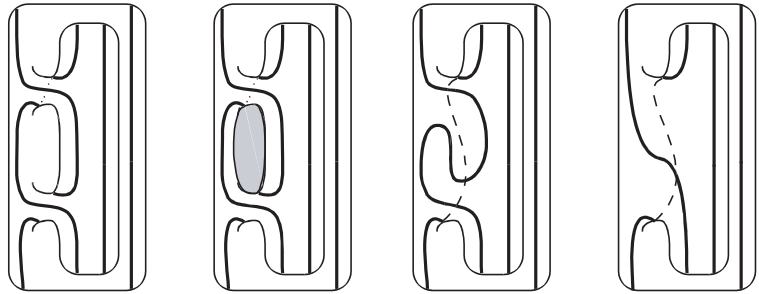}}  
 \end{center}
\caption{Two handles corresponding to crossings of the same sign define a natural decomposition disc. }\label {fig:decomp}
\end{figure}

The innermost curve in the stack corresponds to the front sutured ball, and the condition that each braid generator appears with a unique sign implies that all the handles connecting this ball to the next on have the same sign.   Any pair of one-handles of the same type connecting the same sutured balls defines a natural disc in the surface complement.  (See Figure~\ref{fig:decomp}.)  Decomposing along this disc replaces the pair with a single one-handle of the same sign.  Repeating this as necessary reduces the number of one-handles until the front ball is connected to its neighbor by a single handle.  The suture enters the ball along this one-handle, runs around the ball once, and exits through the handle again.  The suture can thus be isotoped off the front ball and onto the next one,  and the front ball itself collapsed along the one-handle so that  ball that was originally second is now the front ball.  (See Figure~\ref{fig:conn}.)  The process may be repeated until the original sutured manifold is reduced to a single ball with a single connected suture.  Since the complement of a ball in $S^3$ is another ball, this proves that $M$ is taut and the original surface $\mathfrak{A}_{\mathfrak{S}} $ had minimal complexity.
\end{proof}

\begin{figure}
\begin{center}
\scalebox{.6}{\includegraphics{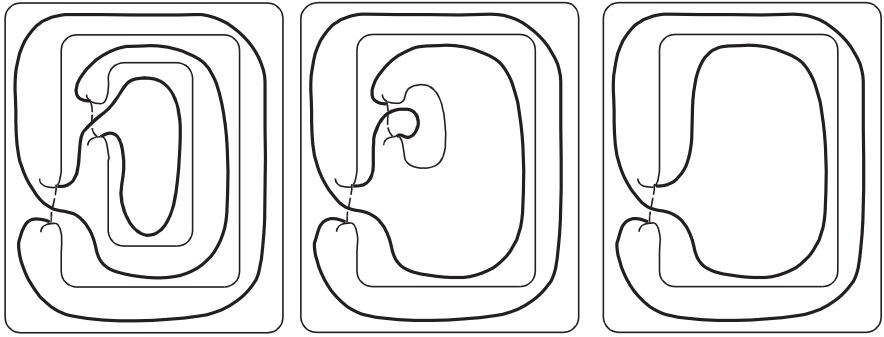}}  
 \end{center}
\caption{The suture may be isotoped off the ball corresponding to the innermost stacked curve, and the ball itself can be collapsed into its neighbor.}\label {fig:conn}
\end{figure}

\section{Extension of Seifert's algorithm to arbitrary classes}\label{sect:arb}
The relative second homology of the complement of a knot has rank one, so any Seifert surface for $K$ generates $H_2(S^3-K, \partial(S^3-K))$.  When $|L|>1$, however, the situation is more interesting, as the number of link components gives the rank of the second homology.  To any fixed second homology class $\sigma \in H_2(S^3-L, \partial(S^3-L))$ one may associate the the minimal complexity of an embedded surface representing $\sigma$.  Thurston first studied  this family of invariants, and recent work of Ozsv\`{a}th, Szab\`{o}, and Ni has shown that the Heegaard Floer link invariant can be used to compute these minimal complexities. (\cite{T}, \cite{OSz4}, \cite{N2})

This paper has been primarily concerned with constructing surfaces representing the Seifert class in the second homology, but the techniques may be extended to constructing surfaces representing arbitrary second homology classes as well.  In particular, $\mathfrak{A}_{\mathfrak{S}}$ may be extended to such an end.  

Fix some  $\sigma \in H_2(S^3-L, \partial(S^3-L))$ by labeling each link component $L_i$ with an integer $n_i$.  As before, begin with an $n$-bridge projection of a link, but instead of a single curve, connect pairs of points by $n_i$ parallel curves.  Alternately, a single curve may be drawn and labeled with its multiplicity; changing the transverse orientation if necessary allows us to assume all $n_i$ are nonnegative.   This frame is a generalization of a crossingless match in the sense that no points are obstructed.  Note that this frame may include isolated points if any of the $n_i$ is zero.  As the braid-induced isotopy acts on this frame,  again decorate the result with obstruction arcs.  Any obstruction by a curve with $n_i > 0$ necessarily creates indirect as well as direct obstructions, so simply resolving a single arc will not return a generalized crossingless match.  

\begin{figure}
\begin{center}
\scalebox{.6}{\includegraphics{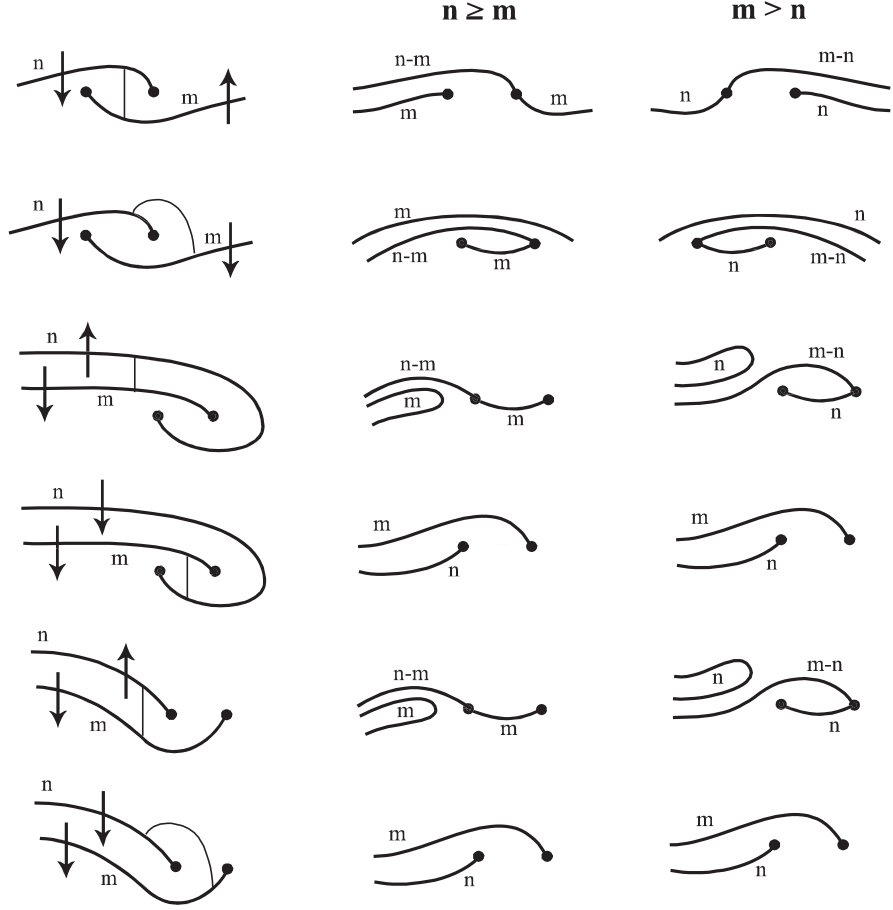}}  
 \end{center}
\caption{The figures on the left depict multicurve obstructions, and the center and right columns show the result of resolving sufficient direct obstructions to yield a generalized crossingless match.  In each case the number of resolutions needed is equal to the multiplicity of the obstructing multicurve.  Note that if $n=m=1$, these diagrams agree with those of Figure~\ref{fig:obstrarc}.}\label{fig:multi}
\end{figure}

Figure~\ref{fig:multi} indicates the number of resolutions and the resulting frames involved in removing multicurve obstructions.  Each of these is derived from repeated application of the resolutions shown in 
Figure~\ref{fig:obstrarc}.  Additionally, one must make sense of obstructed points which are not the endpoints of any curve.  Passing a curve across such a point corresponds to puncturing $\Sigma_{\mathfrak{A}_{\mathfrak{S}}}$ by a component of the link with $n_i=0$, so each puncture increases the complexity by one.  With these tools in hand, one may proceed as in $\mathfrak{A}_{\mathfrak{S}}$, resolving each frame to remove all obstructions before allowing the next elementary isotopy to act.  Once the isotopy has been exhausted, join the generalized crossingless match to the lower generalized crossingless match of the link projection and cap the resulting simple closed curves with two-handles.

This generalization of Seifert's algorithm offers a straightforward method for constructing surfaces representing arbitrary homology classes, and for computing their complexity.  The author is not aware of any alternating projections for which this process fails to construct a minimal-complexity surface, and we speculate that such surfaces might be of minimal complexity for some large class of projections.  

\section{Conclusion}
One may try, beyond the results of Theorems~\ref{thm:alt} and \ref{thm:braid}, to characterize the links for which $\Sigma_{\mathfrak{A}}$ realizes the link genus.  In \cite{Mo}, Moriah used twisted Whitehead doubles to show that the difference between the genus and the canonical genus of a knot may be made arbitrarily large.  (In fact, he showed this difference for genus and free genus, which is a lower bound for canonical genus.)  The two examples of Section~\ref{sect:ex}  are both doubles, although this term is used loosely in the case of Alford's knot  because of the unusual clasp.  We note also an example of Lyon which first established the existence of knots with no incompressible free Seifert surfaces \cite{Ly}.   Although Lyon's knot (****) is not a double, applying $\mathfrak{A}$ to a bridge projection recovers the incompressible, rather than the free, surface.  In each of these cases, the minimal complexity surfaces are composed of a relatively small number of highly twisted and knotted bands.  We suggest that $\mathfrak{A}$ is particularly well-suited to constructing such surfaces, whereas Seifert's algorithm is not.  It would be interesting to find a more precise classification of the type of knot for which these are the minimal complexity Seifert surfaces.

Even when $\mathfrak{A}$ does not produce a minimal complexity surface,  the technique of constructing a Seifert surface via a sequence of Morse slices may still be useful.  In \cite{Li}, the author constructs minimal complexity surfaces in pretzel link complements by performing resolutions different from those suggested by $\mathfrak{A}$.    Thus,  a secondary goal of this paper is to introduce a flexible technique that may be useful in a wide variety of situations.   

\bibliographystyle{alpha}
\bibliography{bibliography.bib}

\end{document}